\documentclass{article}                    
\usepackage{amsfonts}
\usepackage{amsmath}

\newtheorem{theorem}{Theorem}
\newtheorem{acknowledgements}[theorem]{Acknowledgements}

\newtheorem{definition}{Definition}

\newtheorem{lemma}{Lemma}

\newtheorem{proposition}{Proposition}
\newtheorem{remark}{Remark}

\newenvironment{proof}[1][Proof]{\noindent\textbf{#1.} }{\ \rule{0.5em}{0.5em}}

\begin{document}

\title{Recurrence properties of a special type of Heavy-Tailed Random Walk}


\author{P\'{e}ter N\'{a}ndori
}
\date{\vspace{-5ex}}




\maketitle

\begin{abstract}
In the proof of the invariance principle for locally perturbed periodic
Lorentz process with finite horizon,
a lot of delicate results were needed concerning the recurrence
properties of its unperturbed version.
These were analogous to the similar properties of Simple
Symmetric Random Walk. However, in the case of Lorentz process with
infinite horizon, the analogous results for the corresponding random walk
are not known, either. In this paper, these properties are ascertained
for the appropriate random walk (this happens to be in the non normal
domain of attraction of the normal law). As a tool, an estimation of the 
remainder term in the local limit theorem for the corresponding 
random walk is computed.
\end{abstract}

\section{Introduction}
The appearance of the Brownian motion as a limit object in either stochastic
or deterministic models is an extremely important and interesting phenomenon.
The first result in this field is due to M. Donsker (see \cite{Do}) who proved
that the diffusively scaled Simple Symmetric Random Walk (SSRW) converges to 
the Brownian Motion in each dimension. Later, D. Sz\'{a}sz and A. Telcs 
in \cite{Sz-Te} proved that the local perturbation in the integer lattice 
of dimension at least two does not spoil the Brownian limit. \\
In the last decades, a more complex model, i.e. planar periodic Lorentz process was proven
to have Brownian motion, as a limit object. Here, one considers periodically 
situated fixed strictly convex smooth scatterers, and a dimensionless point particle moving among the scatterers and bouncing off at the boundaries according
to the classical law of mechanics (the angle of incidence coincides the angle of
reflection). In the case of finite horizon (i.e. when the free flight vector
of the particle is bounded), diffusive scaling produces Brownian motion (see 
\cite{Bu-Si} and \cite{Bu-Si-Ch}). In the case of infinite horizon (i.e. 
when the free flight vector of the particle is unbounded) a superdiffusive
scaling is needed to obtain the non-trivial Brownian limit (see 
\cite{Sz-Va} and \cite{Ch-Do}).
Again, the question of the effect of local perturbations naturally arises. This
topic has a physical motivation as well, since Lorentz process can be thought
of as the movement of a "classical" electron in a crystal, 
when local perturbation can be some 
impurities or some locally acting external force. The Brownian limit for 
diffusively scaled periodic Lorentz process with finite horizon 
and local perturbation was proven in 
\cite{D-D-K} and \cite{D-D-K2}. Note that here a more involved 
investigation was needed than in the case of SSRW, namely, the wide treatment
of recurrence properties in \cite{D-D-K} was essential. \\
Recently, D. Paulin and D. Sz\'{a}sz 
proved (\cite{Pa-Sz}) that
the random walk, which is very similar to the Lorentz process with infinite
horizon, with local impurities, enjoys the Brownian limit. However, they only
treated some simplified local perturbation (see later), and did not consider
the recurrence properties similar to the ones in \cite{D-D-K}, which 
are expected to be important in the case of infinite horizon, too.
Here, we are going to focus on these recurrence properties. \\
This paper is organized as follows. In Section \ref{secintro}, basic 
definitions, statements are given and another motivation for our
calculations (i.e. the proof of the polynomial decay of the 
velocity auto correlation function for some perturbed random walk) 
is provided. The quite well known local limit theorem for our specific 
type random walk will not be enough for our purposes, i.e. we need to
estimate the remainder term of it. Section \ref{seclimit} is devoted
to this computation. In Section \ref{secrecurrence}, the desired recurrence
properties are obtained, while in Section \ref{secremarks} we give a 
final remark, and indicate a possible direction of further research.

\section{Preliminaries}
\label{secintro}
Let us consider a Random Walk, the behavior of which is close to the one of 
the Lorentz process with infinite horizon.
Namely, define independent random variables $X_{i}$,  such that
\[\mathbb{P} \left(
 X_{i}=n  \right) = c_{1} |n|^{-3}, \]
if $n \neq 0$, and $E _{i}$ to be uniformly distributed on the $4$ unit
vectors in $\mathbb{Z}^2$. Now put $\xi _i = X_i E_i$.
(Here, of course, $c_1 = \frac{1}{2 \zeta (3)}$, but this will
not be important for us.)
Define the Heavy-Tailed Random Walk (HTRW) by $S_{n}:= \sum_{i=1}^{n} \xi{i}$.\\
This distribution is the same, as the one of the free flight vector 
of the Lorentz
process with infinite horizon (see \cite{Sz-Va}). However, one could think that
our choice is rather special, as the walker can only step along the $x$ and
$y$ axis. But this is not the case, as a particle performing 
Lorentz process can have arbitrary long steps only in finitely many directions,
too. Here, we choose that two particular directions, but this is not 
essential.\\
Further, define the one dimensional HTRW as
\[ Q_n := \sum_{i=1}^n X_i.\]
The quite well-known local limit theorem in one dimension states that
\begin{equation}
\label{limit1dim}
\mathbb{P} ( Q_n = x ) \sim \frac{1 }{2 \sqrt{\pi c_1 n \log n} } \exp \left(
- \frac{ x^2 }{4 c_1 n \log n } \right)
\end{equation}
and in two dimensions that
\begin{equation}
\label{limit2dim}
\mathbb{P} ( S_n = x ) \sim \frac{1 }{4 \pi c_1 n \log n } \exp \left(
- \frac{| x |^2 }{4 c_1 n \log n } \right).
\end{equation}

These can found in \cite{Rv}. Later, we will need estimations on the error
terms in (\ref{limit1dim}) and (\ref{limit2dim}), and by computing them,
a proof of (\ref{limit1dim}) and (\ref{limit2dim}) will be provided.\\
Further, we will use the notations
\begin{eqnarray*}
u_2(n) &=& \mathbb{P}(S_n =(0,0)),\\
u_1(n) &=& \mathbb{P}(Q_n =0).
\end{eqnarray*}

In the case of billiards, a quite frequent strategy is to prove exponential
decay of correlations (an interesting result for its own sake) and then to 
use this to prove convergence to the Brownian motion (see \cite{Ch}, for 
instance). As a motivation for our further calculations, we 
are going to illustrate that in the case of local perturbation,
this does not seem to be a good strategy. \\
For this consider the simplest case: a perturbed SSRW ($T_n$) in $\mathbb{Z}^d$,
where perturbation means
that in the origin there is no scatterer, i.e. outside of the origin $T_n$
behaves like an ordinary SSRW, while it flies
through the origin. More precisely,
\begin{equation}
 P \left( T_{n+1} = e_i | T_{n-1}= -e_i, T_n=0 \right) =1, \label{motivegy}
 \end{equation}
where $e_i$ is some neighboring point of the origin in $\mathbb{Z}^d$.
The following Proposition is well known in the physics literature 
(see, for example \cite{Sp}) but surprisingly, I was unable to find a 
mathematical proof for it.

\begin{proposition}
The velocity autocorrelation function of $T_n$ is $O \left( n^{-(d/2+1)}
\right)$.

\end{proposition}

\begin{proof}
First, suppose that $d=1$ and $T_0 =1$. We can identify our process 
with an unperturbed SSRW - $U_n$, say - by simply dropping the origin
and the extra step from it. Formally, define $\tau (n)
= \# \{1 \leq k <n:T_k =0 \}$. Now, if $T_n > 0$, then let $U(n - \tau(n))
= T_n$. If $T_n < 0$, then $U(n - \tau(n))= T_n + 1$.
Now, we have to show that
\begin{eqnarray*}
&& \mathbb{P}(U(2n)=0, U(2n+1) =1) - \mathbb{P}(U(2n+1)=1, U(2n+2) =0) \\
&=& \frac{1}{2}\left[  \mathbb{P}(U(2n)=0) - \mathbb{P}(U(2n+1)=1) \right] = 
O \left( n^{-3/2}\right),
\end{eqnarray*}
which is an elementary consequence of the 
well known Edgeworth expansion. \\
Now, suppose that $d>1$ and $T_0 = (1,0,...,0)$.
It suffices to prove
\begin{equation}
\int_{\Omega} I_{ \{T_n = T_0 \} } - I_{ \{T_n = -T_0 \} } dP =
O \left( n^{-(d/2+1)} \right). \label{motivketto}
\end{equation}
Let $V$ be the orthogonal complement space of $T_0$ and define
\begin{eqnarray*}
H = \{ \omega: (V \setminus 0) \cap \{ T_0, ..., T_n\} \neq \emptyset
\} \subset \Omega.
\end{eqnarray*}
Because of the reflection principle, the part of the integral in 
(\ref{motivketto})
over $H$ is zero. The integral over $\Omega \setminus H$ can be treated
similarly, as it was done in the one dimensional case.

\end{proof}


\section{Local Limit Theorem with Remainder Term}
\label{seclimit}

The aim of this section is to estimate remainder term in the limit theorem (\ref{limit2dim}).
To do this, first we have to deal with the one dimensional case. Similar calculations were
done previously, see, for example \cite{Ha-Pe} and \cite{Jo-Pa}.
However, in these articles only one dimensional, non-lattice distributions
were considered. Fortunately, we do not need precise calculation of the
remainder term, i.e. summability is enough for our purposes. As usual, we
start with the computation of the characteristic function.

\begin{lemma}
\label{lemma1}
For the characteristic function $\phi$ of $X_1$
\[ \phi(t) = 1-2c_1 t^2 |\log |t| | + O \left( t^2 \right), \]
as $t \rightarrow 0$.
\end{lemma}

\begin{proof}
Since the distribution is symmetric, it suffices to prove for $t>0$.
Fix $\varepsilon >0$ such that $1-x^2-x^3 < \cos x < 1-x^2+x^3$ if
$|x| < \varepsilon$. Now, let us consider the decomposition
\[ \phi(t) = \mathbb{E} (\exp (itX)) = \sum_{n=1}^
{\varepsilon \lfloor t^{-1} \rfloor}
\frac{2 c_1}{n^3} \cos (tn) +
\sum_{n = \varepsilon \lfloor t^{-1} \rfloor +1}^{\infty}
\frac{2 c_1}{n^3} \cos (tn) =: S_1 + S_2. \]

It is easy to see that
\[ S_2 =2c_1  \int_{\varepsilon}^{\infty} \frac{\cos x}{x^3} dx  t^2 + o
\left( t^2 \right) = O\left( t^2 \right). \]

On the other hand, since
\[ S_1 = 2 c_1 \sum_{m=t, m \in t\mathbb{Z}}^{\varepsilon}
t^3 m^{-3} \cos m,\]
we have
\[ \left| \frac{S_1}{2 c_1} -
\sum_{m=t, m \in t\mathbb{Z}}^{\varepsilon} t^3 m^{-3}
+ \sum_{m=t, m \in t\mathbb{Z}}^{\varepsilon} t^3 m^{-1} \right|
< \sum_{m=t, m \in t\mathbb{Z}}^{\varepsilon} t^3.\]

Now the estimations
\begin{eqnarray*}
\sum_{m=t, m \in t\mathbb{Z}}^{\varepsilon} t^3 m^{-3} &=& \frac{1}{2 c_1}
+ O \left( \frac{t^2}{\varepsilon ^2} \right)\\
\sum_{m=t, m \in t\mathbb{Z}}^{\varepsilon} t^3 m^{-1} &=&
t^2 log \left( \frac{\varepsilon}{t} \right) + O \left( t^2 \right)
\end{eqnarray*}
and
\[
\sum_{m=t, m \in t\mathbb{Z}}^{\varepsilon} t^3 = O(t^2)\]
finish the proof.

\end{proof}

Now, we turn to the estimation of the remainder term in the one dimensional
local limit theorem.

\begin{theorem}
\label{theoremlocal1}
For the one dimensional HTRW the following estimation holds uniformly
in $x$
\[ \mathbb{P} ( Q_n = x ) - \frac{1}{\sqrt{2 \pi} \sqrt{2 c_1}
\sqrt {n \log n} }
\exp \left( - \frac{x^2}{4 c_1 n \log n} \right) =
O \left( \frac{\log \log n}{\sqrt {n \log^3 n} } \right) \]
\end{theorem}

\begin{proof}
Let $g$ denote the probability density function of the standard Gaussian law.
Then we have
\[ g(z) = \frac{1}{2 \pi} \int_{- \infty}^{\infty} \exp \left(
-izs - \frac{s^2}{2} \right) ds.\]
On the other hand, according to the Fourier inversion formula,
\[ \mathbb{P} ( Q_n = x ) =  \frac{1}{2 \pi} \int_{- \pi}^{\pi}
\exp \left( -itx \right) \phi^n (t) dt.\]
By an elementary argument (see, for example, \cite{Ib-Li})
our result follows from the statement
\begin{equation}
\label{limthm1}
\left|  \sqrt { 2 c_1 n \log n} \frac{1}{2 \pi}\int_{- \pi}^{\pi}
\exp (-itx) \phi^n (t) dt-
g \left( \frac{x}{\sqrt { 2 c_1 n \log n}} \right) \right| = O \left(
 \frac{\log \log n}{\log n} \right),
 \end{equation}
 where the great order on the right hand side is uniform in $x$.
 As it is quite usual in the theory of limit
 theorems (see again \cite{Ib-Li}), we estimate the left
 hand side of (\ref{limthm1}) by the sum of several integrals
 \begin{eqnarray*}
 &&\int_{\frac{1}{\log n} < |s| < \log  n} \left| \phi^n \left(
 \frac{s}{\sqrt { 2 c_1 n \log n} } \right) - \exp \left(
 - \frac{s^2}{2} \right) \right| ds \\
 &+& \int_{|s| < \frac{1}{\log n}} \left| \phi^n \left(
 \frac{s}{\sqrt {2 c_1 n \log n} } \right) \right| ds
 + \int_{|s| < \frac{1}{\log n}} \left|  \exp \left(
 - \frac{s^2}{2} \right) \right| ds \\
 &+& \int_{\log n < |s| < \gamma \sqrt {2 c_1 n \log n} } \left|
 \phi^n \left( \frac{s}{\sqrt {2 c_1 n \log n} } \right) \right| ds \\
 &+& \int_{\gamma \sqrt {2 c_1 n \log n} < |s| < \pi \sqrt {2 c_1 n \log n}}
 \left| \phi^n \left( \frac{s}{\sqrt {2 c_1 n \log n} } \right) \right| ds \\
 &+& \int_{\log n < |s| } \left|  \exp \left(
 - \frac{s^2}{2} \right) \right| ds
 =: I_1 + I_2 +I_3 +I_4 +I_5 +I_6.
 \end{eqnarray*}
So it suffices to prove that $I_j =O \left(
 \frac{\log \log n}{\log n} \right)$, for $j \in \{ 1,2,3,4,5,6 \} $. \\
 For the estimation of $I_1$, observe that for $\frac{1}{\log n} < |s|
 < \log n$ Lemma \ref{lemma1} yields
 \[ \phi^n \left( \frac{s}{\sqrt {2 c_1 n \log n} } \right) =
 \exp \left( - \frac{s^2}{2} \right) \left[ 1 + O \left(
  \frac{s^2 \log \log n}{\log n} \right) \right], \]
  where the great order on the right hand side is uniform in $s$. Hence
  \[ I_1 < \int_{\frac{1}{\log n} < |s| < \log  n}
  s^2 \exp \left( - \frac{s^2}{2} \right) ds O \left(
  \frac{\log \log n}{\log n} \right) = O \left(
  \frac{\log \log n}{\log n} \right).\]
  Further, $|\phi(t)| \leq 1$ yields $I_2 = O \left(
  \frac{\log \log n}{\log n} \right)$ and $I_3 = O \left(
  \frac{\log \log n}{\log n} \right)$ is trivial.
  It can be proven (see Theorem 4.2.1. in \cite{Ib-Li}) that
  there exists $\gamma >0$ such that
  \[ \phi ^n  \left( \frac{s}{\sqrt {2 c_1 n \log n} } \right)
  < \exp \left( -C |s| \right),\]
  with an appropriate $C$ if $|t| < \gamma$.
  This estimation implies $I_4 < O \left(
  \frac{\log \log n}{\log n} \right)$. Observe that $|\phi(t)| \leq 1$
  and $|\phi(t)| = 1$ holds if and only if $t \in 2 \pi \mathbb{Z}$.
  As $|\phi(t)|$ is continuous in $t$, there exists some $C' <1$
  such that $|\phi(t)| < C'$ for $t \in [\gamma, \pi]$. It follows
  that $I_5 < O \left( \frac{\log \log n}{\log n} \right)$. Finally,
  $I_6 < O \left( \frac{\log \log n}{\log n} \right)$ by elementary
  computation. Hence the statement.

  \end{proof}

  Now, we turn to the two dimensional case. Define the two dimensional
  characteristic function $\phi_2: \mathbb{R}^2
  \rightarrow \mathbb{C}$, $\phi_2 (t) = \mathbb{E} (\exp (it' \xi_1))$,
  where $'$ stands for transpose, and write $t=(t_1, t_2)', s=(s_1,s_2)'$.
  Lemma \ref{lemma1} implies that
  \[ \phi_2 (t) = 1- c_1 t_1^2 |\log |t_1|| - c_1 t_2^2 |\log |t_2|| +
  O \left( |t|^2 \right),\]
  as $|t| \rightarrow 0$.
  Similarly to the one dimensional case, the local limit theorem with
  remainder term reads as follows.

\begin{theorem}
\label{theoremlocal2}
For the two dimensional HTRW the following estimation holds uniformly
for $x \in \mathbb{R}^2$
\[ \mathbb{P} ( S_n = x ) - \frac{1}{{2 \pi} 2 c_1 n \log n }
\exp \left( - \frac{| x |^2}{4 c_1 n \log n} \right) =
O \left( \frac{\log \log n}{ {n \log^2 n} } \right) \]

\end{theorem}

\begin{proof}
The proof is similar to the proof of Theorem \ref{theoremlocal1}.
Let $g$ denote the probability density function of the
two dimensional standard Gaussian law.
Then we have
\[ g(z) = \frac{1}{(2 \pi)^2} \int_{- \infty}^{\infty}
\int_{- \infty}^{\infty}
\exp \left(
-it'z - \frac{t't}{2} \right) dt.\]
On the other hand, according to the Fourier inversion formula
\[ \mathbb{P} ( S_n = x ) =  \frac{1}{(2 \pi)^2} \int_{- \pi}^{\pi}
\int_{- \pi}^{\pi}
\exp \left( -it'x \right) \phi_2^n (t) dt.\]
Just like previously, it is enough to prove that
\begin{equation}
\label{limthm2}
\left|  2 c_1 n \log n \frac{1}{(2 \pi)^2}\int_{- \pi}^{\pi}
\int_{- \pi}^{\pi} \exp (-it'x) \phi_2^n (t) dt-
g \left( \frac{x}{\sqrt { 2 c_1 n \log n}} \right) \right|
  \end{equation}
 is in $O \left(  \frac{\log \log n}{\log n} \right)$. 
 The analogue of the previous decomposition in the present case is
 \begin{eqnarray*}
  &&\int_{\frac{1}{\log^3 n} < |s_1|,|s_2| < \log  n} \left| \phi_2^n \left(
   \frac{s}{\sqrt { 2 c_1 n \log n} } \right) - \exp \left(
    - \frac{s's}{2} \right) \right| ds \\
     &+& 2 \int_{|s_1| < \frac{1}{\log^3 n}
      \& |s_2| < \log n} \left| \phi_2^n \left(
       \frac{s}{\sqrt {2 c_1 n \log n} } \right) \right| ds \\
        &+& \int_{|s_1| < \frac{1}{\log^3 n}
	  \& |s_2| < \log n}\left|  \exp \left(
	   - \frac{s's}{2} \right) \right| ds \\
	    &+& \int_{\log  n < |s| < \gamma \sqrt {2 c_1 n \log n} } \left|
	     \phi_2^n \left( \frac{s}{\sqrt {2 c_1 n \log n} } \right) \right| ds \\
	      &+& \int_{\gamma \sqrt {2 c_1 n \log n} < |s| < \pi \sqrt {2 c_1 n \log n}}
	       \left| \phi_2^n \left( \frac{s}{\sqrt {2 c_1 n \log n} } \right) \right| ds \\
	        &+& \int_{\log  n < |s| } \left|  \exp \left(
		 - \frac{s's}{2} \right) \right| ds
		  =: I_1 + I_2 +I_3 +I_4 +I_5 +I_6.
		   \end{eqnarray*}
		    So it suffices to prove that $I_j =O \left(
		      \frac{\log \log n}{\log n} \right)$, for $j \in \{ 1,2,3,4,5,6 \} $. \\
		      All the above integrals can be estimated as it was done in the proof of
		      Theorem \ref{theoremlocal1} except for $I_4$. For the latter, we adapt
		      the argument of Rvaceva (see \cite{Rv}).
		      It is easy to see that
		      \[ \frac{\Re \log \phi_2 (a t)}{\Re \log \phi_2 (t)} \rightarrow a^2 \]
		      as $|t| \rightarrow 0$
		      (here $\Re$ denotes real part). Hence, for $\gamma$ small enough,
		      \[ \Re \log \phi_2 (t) > e \Re \log \phi_2 (t/e) \]
		      holds for $|t| < \gamma$.
		      Now, pick $k \in \mathbb{N}$ such that
		      $\exp(k) \leq \gamma \sqrt {2 c_1 n \log n} < \exp(k+1)$ and write
		      \begin{eqnarray*}
		      &I_4& \leq \sum_{m= \log \log n}^{k} \int_{ \exp (m)<|s|< \exp (m+1)}
		      \left| \phi_2^n \left( \frac{s}{\sqrt {2 c_1 n \log n} } \right) \right|
		      ds \\
		      && <  \sum_{m= \log \log n}^{k} \exp(2m) \int_{ 1<|s|< e} \exp \left(
		      n \exp(m) \Re \log
		      \phi_2 \left( \frac{s}{\sqrt {2 c_1 n \log n} } \right) \right) ds.
		      \end{eqnarray*}
		      The argument used in the estimation of $I_1$ implies that
		      \[ n \Re \log \phi_2 \left( \frac{s}{\sqrt {2 c_1 n \log n} } \right) =
		      -\frac{|s|^2}{2} +o(1)\]
holds uniformly for $s \in [1,e]$, whence for some $C'<1$
\[ I_4 < \sum_{m= \log \log n}^{k} \exp(2m) (e-1) C'^{ \exp(m)}.\]
So we proved $I_4 =O (\frac{1}{\log n})$, hence the statement.

\end{proof}

\section{Recurrence properties}
\label{secrecurrence}

In this section we discuss the recurrence properties of $S_n$ and $Q_n$ that 
are supposed to be important in the case of billiards, too (note that these
are analogous to the ones considered in \cite{D-D-K}). For SSRW, these kind
of results were proven in \cite{Er-Ta} and \cite{Da-Ka}. We begin with the
two dimensional case.

\begin{definition}
Let $\tau_2$ be the first return to the origin in two dimensions, i.e.
\[ \tau_2 = \min \{ n > 0: S_n = (0,0) \} \]
\end{definition}

\begin{theorem}
\label{tetel1}
$\mathbb{P} (\tau_2 >n) \sim \frac{4 \pi c_1}{  \log \log n}$
\end{theorem}

\begin{theorem}
\label{thm:rec2}
Let $N_2^n = \# \{k \leq n : S_k = (0,0) \}$. 
Then
\[ \frac{ N_2^n}{ \log \log n} \]
converges to an exponential random variable with expected value 
$\frac{1}{4 \pi c_1}$.
\end{theorem}

Theorem \ref{tetel1} and Theorem \ref{thm:rec2} can be easily proven 
combining the original proofs 
(see \cite{Dv-Er} and \cite{Er-Ta}) with (\ref{limit2dim}).\\

\begin{definition}
Let $t_v$ be the hitting time of the origin, starting from the site
$v \in \mathbb{Z}^2$, i.e.
\[ t_v = \min \{ k \geq 0: S_k =(0,0) | S_0 = v\}.\]
\end{definition}

The following recurrence property is less known but is of crucial
importance in the argument of \cite{D-D-K2}.

\begin{theorem}
\[ \frac{\log \log t_v}{\log \log | v | } \Rightarrow \frac{1}{U} \]
as $|v| \rightarrow \infty$,
where $U$ is uniformly distributed on $[0,1]$ and $\Rightarrow$ 
stands for weak convergence.
\end{theorem}

\begin{proof}
We adapt the proof of \cite{Er-Ta}. Let
\[ \zeta (x,n) = \# \{ 1 \leq k \leq n: S_k = x\} \]
be the local time of the walk at site $x$ up to time $n$ and
\[ \gamma(n) = \mathbb{P} \left( \tau_2 > n \right).\]
Further, we will need the estimation on the remainder term of
the local limit theorem. More precisely, we will use the following 
estimation
\begin{equation}
\label{nehezminuszegy}
\mathbb{P} \left( S_n = y \right) = \frac{1}{4 \pi c_1 n \log n} -
|y|^2 O \left( \frac{1}{n^2 \log^2 n} \right) + O
 \left( \frac{\log \log n}{n \log^2 n} \right),
\end{equation}
where the great orders are uniform in $\{ y: |y| < \sqrt {n \log n} \}$.
Note that (\ref{nehezminuszegy}) is a consequence of Theorem 
\ref{theoremlocal2}.
We are going to prove the following assertion. \\
If we choose $x_n \in \mathbb{Z}^2$ such that 
\[|x_n| \sim \exp \left( \frac{1}{2} \log ^\delta n \right) \]
for some fix $0< \delta <1 $, then
\begin{equation}
\label{nehez}
\mathbb{P} \left( \zeta \left( x_n, n\right) = 0\right) \rightarrow
\delta,
\end{equation}
as $n \rightarrow \infty$.
It is easy to see that (\ref{nehez}) implies the statement of the theorem.\\
As in \cite{Er-Ta}, we consider the identities
\begin{equation}
\label{nehezegy}
\sum_{i=0}^n u_2(i) \gamma(n-i) =1
\end{equation}
and
\begin{equation}
\label{nehezketto}
\mathbb{P} \left( \zeta \left( x_n, n\right) = 0\right) 
+ \sum_{i=1}^n \mathbb{P} \left( S_i = x_n \right)  \gamma(n-i)=1.
\end{equation}
Combining (\ref{nehezegy}) and (\ref{nehezketto}) we obtain
\begin{equation}
\label{nehezharom}
\mathbb{P} \left( \zeta \left( x_n, n\right) = 0\right) - \gamma(n) =
\sum_{i=1}^n \left( u_2(i) -
\mathbb{P} \left( S_i = x_n \right) \right)  \gamma(n-i).
\end{equation}
Using the fact that $\gamma$ is monotonic, Theorem \ref{tetel1} and 
the estimation (\ref{nehezminuszegy}) we conclude that 
the right hand side of (\ref{nehezharom}) is smaller than
\begin{eqnarray*}
&&\frac{4 \pi c_1 + o(1)}{\log \log n} \sum_{k=1}^{
\exp \left( \log^{\delta} n\right)} \frac{1}{4 \pi c_1 k \log k}\\
&+& \frac{4 \pi c_1 + o(1)}{\delta \log \log n} \sum_{ k=
\exp \left( \log^{\delta} n\right)}^{\sqrt n} |x_n|^2 O
\left( \frac{1}{k^2 \log^2 k} \right)
+ \sum_{k = \sqrt n}^{n} |x_n|^2 O \left( \frac{1}{k^2 \log^2 k} \right)\\
&+& \sum_{ k=
\exp \left( \log^{\delta} n\right)}^{ \infty} O\left(
\frac{\log \log k}{k \log^2 k}\right) = \delta + o(1).
\end{eqnarray*}
So we arrived at the upper bound. For the lower bound define 
\[k_1 = \frac{\exp \left( \log^{\delta} n \right)}{\log n}. \]
Theorem \ref{tetel1} and Theorem \ref{theoremlocal2} imply that the right 
hand side of (\ref{nehezharom}) is bigger than
\begin{eqnarray*}
&&\gamma(n) \sum_{k=1}^{k_1} \left[ u_2(k)- \mathbb{P}
(S_k = x_n) \right] \geq \\
&&\frac{4 \pi c_1 + o(1)}{\log \log n} \sum_{k=1}^{k_1} \Big[
\frac{1}{4 \pi c_1 k \log k} + O\left(
\frac{\log \log k}{k \log^2 k}\right) \Big] \\ 
&&+ \frac{4 \pi c_1 + o(1)}{\log \log n} \sum_{k=1}^{k_1} \Big[
- \frac{1}{4 \pi c_1 k \log k}
\exp \left( - \frac{|x_n|^2}{4 c_1 k \log k} \right) \Big] \\
&>& \delta +o(1) + \frac{O(1)}{\log \log n} 
- O\left( \frac{1}
{\log \log n} \right) \exp \left( - \frac{|x_n|^2}{k_1 \log k_1}\right)
\sum_{k=1}^{k_1} \frac{1}{k \log k} \\
&>& \delta + o(1).
\end{eqnarray*}
Thus we have proved (\ref{nehez}). The statement follows.
\end{proof}

\begin{remark}
Note that for the adaptation of the Erd\H{o}s-Taylor type argument for our
setting, the summability of the remainder term in the local limit theorem -
i.e. Theorem \ref{theoremlocal2} -
was essential. The situation was basically the same in \cite{Na}, however,
in a different context.
\end{remark}

It would be interesting to find an intuitive reason for the appearance of
the exponential and the uniform distributions as limit laws. However, neither
Erd\H{o}s and Taylor gave explanation in \cite{Er-Ta}, nor the present
author can give any. Now, we turn to the one dimensional case.

\begin{definition}
Let $\tau_1$ be the first return to the origin in one dimension, i.e.
\[ \tau_1 = \min \{ n > 0: Q_n = 0 \} \]
\end{definition}

\begin{theorem}
\label{tail1dim}
$\mathbb{P} (\tau_1 >n) \sim \frac{2 \sqrt {c_1}}{ \sqrt {\pi}} 
\sqrt{ \frac{\log n}{n} }$
\end{theorem}

\begin{proof}
Theorem \ref{tail1dim} can be easily proven by the usual way. One has
to consider the renewal equation
\[ \sum_{k=0}^n u_1(k) \mathbb{P} (\tau_1 >n-k) = 1, \]
and the identity
\[ U(x)V(x) = \frac{1}{1-x},\]
where
\begin{eqnarray*}
U(x) &=& \sum_{k=0}^{\infty} u_1(k)x^k \\
V(x) &=& \sum_{k=0}^{\infty} \mathbb{P} (\tau_1 >k)x^k.
\end{eqnarray*}
Now, the well known Tauberian theorem (Theorem XIII.5. in \cite{Fe})
implies that
\[ U(x) \sim \frac{1}{\sqrt{1-x}} \frac{1}{\sqrt{\pi c_1}} \Gamma
\left( \frac{3}{2} \right) \frac{1}{\sqrt{ \log \frac{1}{1-x}}}\]
as $x \rightarrow 1$, thus
\[ V(x) \sim \frac{1}{\sqrt{1-x}} \frac{\sqrt{\pi c_1}}{\Gamma
\left( \frac{3}{2} \right)} \sqrt{\log \frac{1}{1-x}}\]
as $x \rightarrow 1$. Since $\mathbb{P} (\tau_1 >n)$ is monotonic in $n$,
the previous Tauberian theorem infers the statement.
\end{proof}

\begin{theorem}
Let $N_1^n = \# \{k \leq n : Q_k = 0 \}$. Then
\[ \frac{N_1^n \sqrt {\log n}}{  \sqrt n} \]
converges to a Mittag-Leffler distribution with parameters $1/2$
and $(2 \sqrt{ c_1})^{-1}$, i.e. to the distribution, the $k^{th}$ moment
of which is 
\[\frac{1}{(2 \sqrt{ c_1})^k} \frac{k!}{
\Gamma \left( \frac{k}{2} +1\right)}. \]
\end{theorem}

\begin{proof}
As in the case of \cite{D-D-K}, it suffices to prove that for $k$ fix:
\begin{eqnarray}
\label{harom}
 \sum_{n_i \geq 3, n_1 + n_2 + ... + n_k \leq n} \prod_j \frac{1}{ \sqrt{ n_j \log n_j }} \sim \frac{n^{k/2} }{ \log^{k/2} n} \frac{ \Gamma (1/2)^k}{ \Gamma (k/2 + 1) }.  
\end{eqnarray}

Note that $\Gamma(1/2) = \sqrt \pi$. Elementary calculations show that 
(\ref{harom}) holds for $k=1$.
For $k > 1$ define
\begin{eqnarray*}
\mathcal{H}_1 &=& \{ n_i \geq \frac{n}{ \log n}, n_1 + n_2 + ... + n_k \leq n \}  \\
\mathcal{H}_2 &=& \{ n_i \geq \frac{ \sqrt n}{ \log n }, \exists j : n_j < \frac{n}{ \log n }, n_1 + n_2 + ... + n_k \leq n  \} \\
\mathcal{H}_3 &=& \{ n_i \geq 3, \exists j : n_j < \frac{ \sqrt n}{ \log n }, n_1 + n_2 + ... + n_k \leq n \}
\end{eqnarray*}

Now, split the sum in (\ref{harom}) into three parts, sums over $\mathcal{H}_i$'s, $1 \leq i \leq 3$. \\
Define $s_j = n_j / n$ and observe that
\[  \frac{1}{ \sqrt{ n_j \log n_j }} =  \frac{1}{ \sqrt{ s_j }}  \frac{1}{ \sqrt{ n }}  \frac{1}{ \sqrt{ \log{s_j} + \log n }}.   \]
Since $\log{s_j} + \log n =  (1+o(1)) \log n$ uniformly in $\mathcal{H}_1$, 
it is not difficult to deduce that 
\begin{eqnarray*}
&&\sum_{(n_1,n_2,...,n_k) \in \mathcal{H}_1} \prod_j \frac{1}{ \sqrt{ n_j \log n_j }} \\
&\sim&  \frac{n^{k/2} }{ \log^{k/2} n}
\int ... \int_{0 <t_1 < t_2 < ... <t_k<1} \frac{1}{ \sqrt {t_1}}  \frac{1}{ \sqrt {t_2-t_1}}
... \frac{1}{ \sqrt {t_k-t_{k-1}}} dt_1 ... dt_k \\
&=& \frac{n^{k/2} }{ \log^{k/2} n} \frac{ \Gamma (1/2)^k}{ \Gamma (k/2 + 1) }.
\end{eqnarray*}

For the sum over $\mathcal{H}_2$, consider the case when
$\frac{ \sqrt n}{ \log n } < n_1 < \frac{n}{ \log n }$ and
$ n_i > \frac{n}{ \log n }$ for $ 2 \leq i$ 
(other cases can be treated similarly).
Now, $\log{s_1} + \log n >  (1/2+o(1)) \log n$ and
$\log{s_i} + \log n =  (1+o(1)) \log n$ for $2 \leq i$, uniformly. Thus,

\begin{eqnarray*}
&& \sum_{ \frac{ \sqrt n}{ \log n } < n_1 < \frac{n}{ \log n }, 
n_i > \frac{n}{ \log n }: 2 \leq i } \prod_j \frac{1}{ \sqrt{ n_j \log n_j }} \\
&<&
(\sqrt 2 + o(1)) \sum_{ \frac{ \sqrt n}{ \log n } < n_1 < \frac{n}{ \log n },
n_i > \frac{n}{ \log n }: 2 \leq i } \prod_j \frac{1}{ \sqrt{ s_j n \log n }} 
< 2 \frac{n^{k/2}}{ \log^{k/2} n} o(1).
\end{eqnarray*}
For the third sum, the proof goes by induction on $k$.
Assuming that (\ref{harom}) holds for $k-1$, one has
\[
\sum_{(n_1,n_2,...,n_k) \in \mathcal{H}_3} \prod_{j=1}^k \frac{1}
{ \sqrt{ n_j \log n_j }} < k \frac{\sqrt n}{\log n}
\sum_{n_1+n_2+...+n_{k-1} \leq n } \prod_{j=1}^{k-1} \frac{1}
{ \sqrt{ n_j \log n_j }},
\]
which is $o \left( \frac{n^{k/2}}{ \log^{k/2} n} \right)$.
(\ref{harom}) follows. 

\end{proof}

\section{Final remark}
\label{secremarks}

As it was mentioned in the Introduction, in the case of Lorentz process
with infinite horizon, another type of 'recurrence' can happen.
Namely, if a scatterer is moved into a corridor (here corridor means infinite
trajectories without collision), then there are arbitrary long flights where
in the periodic Lorentz process there would not be collision, while 
in the perturbed one there are some. In the random walk context, 
it can happen that the unperturbed 
walk would fly over the origin, while the perturbed one has to stop. Note that
this phenomenon is evitable if one considers finite horizon, or in
the case of infinite horizon just shrinks one of the scatterers as a 
perturbation. However, the same behavior (i.e. the Brownian
limit with the same scaling) is conjectured in this general 
perturbation, as well. The aim of the following computation is to give
some reason for this conjecture. As the constants do not play 
important role in the sequel, they will not be computed 
and every appearance of $C$ may denote different constant.\\
Define
\[ a_{n} = \mathbb{P} ((0,0) \in \overline{S_{n},S_{n+1}}, (0,0) \neq S_n ) \]
to be the probability of the event that step $n+1$ flies over the origin.
Observe that
\[ a_{n} = \frac 12  \mathbb{P} \left( (S_{n})_{1} = 0, |X_{n+1}| 
\geq |(S_{n})_{2}| \right), \]
where $(S_{n})_{i}$ denotes the $i^{th}$ coordinate of $S_n$.
The local limit theorem implies
 $ a_{n} < C \frac{1}{\sqrt{n \log n}} b_n$,  where
\[ b_n = \mathbb{P} \left( |X_{n+1}| \geq |(S_{n})_{2}| \right). \]
For the estimation of $b_n$ observe that if  $|(S_{n})_{2}| > d_n$, then
$b_n$ is bounded by $C \sum_{k=d_n}^{ \infty } k^{-3} =
O(d_n^{-2})$. On the other hand, the probability of $|(S_{n})_{2}|$ 
being smaller than $d_n$ is roughly estimated by
$O(d_n \frac{1}{\sqrt{n \log n}})$. Thus
\[ b_n = O(d_n^{-2}) + O(d_n \frac{1}{\sqrt{n \log n}}) = 
O \left( (n \log n  )^{-1/3} \right), \]
whence
\[  a_n = O \left( (n \log n)^{-5/6} \right). \]
If $\rho _n$ denotes the number of jumps over the origin up to time $n$
and $\theta _n =\mathbb{E} (\rho _n)$, then we have just proved
\[ \theta _{n} =o(n^{1/6}). \]
Note that in the case of \cite{Sz-Te} and \cite{Pa-Sz} the key observation
was that the time spent at the perturbed area up to $n$ is much smaller than
$\sqrt n$.
That is why it is reasonable to expect the same Brownian limit in the case
of such perturbation, where we introduce some nice further step at the
time of flying over the origin, too. Here nice means that presumably the
step distribution should have some finite moment of order $\varepsilon$.
This could be subject of future research.

\begin{acknowledgements}
Special thanks are due to Domokos Sz\'{a}sz and
Tam\'{a}s Varj\'{u} as this work was inspired by personal communications
with them. The author is also grateful to Andr\'{a}s Telcs for reading the 
manuscript and making useful comments.
\end{acknowledgements}

\end{document}